\documentclass[12pt,reqno]{amsart}

\usepackage[english]{babel}
\usepackage[latin1]{inputenc}
\usepackage{amsfonts}
\usepackage{amssymb}
\usepackage{amsmath}
\usepackage{amsthm}
\usepackage{graphicx}
\usepackage{epsfig}
\usepackage{fancyhdr}
\usepackage[all]{xy}

\numberwithin{equation}{section}

\newcommand{\F}{\mathbb{F}}
\newcommand{\C}{\mathbb{C}}
\newcommand{\R}{\mathbb{R}}

\newtheorem{theorem}{Theorem}[section]

\newtheorem{prop}[theorem]{Proposition}

\newtheorem{exa}[theorem]{Example}

\begin{document}

\title[Action of Complex Symplectic Matrices]{Action of Complex Symplectic Matrices on the Siegel Upper Half Space}

\author[Acharya]{Keshav Raj Acharya}
\address{Department of Mathematics,
Embry-Riddle Aeronautical University,
600 S. Clyde Moris Blvd., Daytona Beach, FL 32114, U.S.A.}
\email{acharyak@erau.edu}

\author[McBride]{Matt McBride}
\address{Department of Mathematics and Statistics,
Mississippi State University,
175 President's Cir., Mississippi State, MS 39762, U.S.A.}
\email{mmcbride@math.msstate.edu}

\thanks{}

\date{\today}

\begin{abstract} 
The Siegel upper half space, $\mathcal{S}_n$, the space of complex symmetric matrices, $Z$ with positive definite imaginary part, is the generalization of the complex upper half plane in higher dimensions.  In this paper, we study a generalization of linear fractional transformations, $\Phi_S$, where $S$ is a complex symplectic matrix, on the Siegel upper half space. We partially classify the complex symplectic matrices for which $\Phi_S(Z)$ is well defined. We also consider $\mathcal S_n$ and $\overline{\mathcal S}_n$ as metric spaces and discuss distance properties of the map $\Phi_S$ from $\mathcal S_n$ to $\mathcal{S}_n$ and $\overline{\mathcal S}_n$ respectively.
	
%
	%

\end{abstract}

\maketitle

\section{Introduction}
The Siegel upper half-space which we denote by  $\mathcal{S}_n$ is the set of all $n\times n$ symmetric matrices $Z$ of the form: $Z=X+iY$ where $X$ and $Y$ are symmetric matrices and $Y$ is positive definite. When $n =1$, then
\begin{equation*}
\mathcal{S}_1 = \C^{+} =\{z\in \C : z= x+iy, \,\, y>0 \}\ .
\end{equation*}  
Therefore the space $\mathcal{S}_n$ is a generalization of the complex upper half plane. For more on details on the Siegel space see \cite{Sie}. In this paper we study the action of symplectic group $SP_{2n}(\C)$ on  $\mathcal S_n$. \\ 

The symplectic group $SP_{2n}(\C)$ is the group of all $2n\times 2n$ complex matrices $S$ satisfying $S^tJS =J$ with 

\begin{equation*} 
J=
\begin{pmatrix}
O & I_n \\ -I_n & O 
\end{pmatrix}
\end{equation*} 
where $I_n$ is the $n\times n$ identity matrix and $O$ is the $n\times n$ zero matrix. 
  	 
When restriced to $\R$, the symplectic group $ SP_{2n}(\R)$ is a generalization of the group $SL_2(\R)$ to higher dimension. The action that we consider here has similarities with the action of $SL_2(\R)$ on the hyperbolic plane $\C^+$. A study of this action was done by Carl Ludwig Siegel in 1943 and published his work in the book ``Symplectic Geometry" in which not only the analytical and geometrical aspects of the action are considered but also some applications to number theory. In 1999, Freitas gave more insight into this theory from the point of view of topological spaces, see \cite{F} for details. In particular he compactified the space and extended the action continuously to a compact domain. There are several studies on the Siegel upper half space motivated from other areas, for example \cite{Oh, Jose}. However, the action of a complex symplectic matrix $SP_{2n}(\C)$ has yet to be studied. Here we consider complex symplectic matrix with non zero imaginary part. \\

To relate the symplectic group to the Siegel upper half space, we write a symplectic matrix $S \in SP_{2n}(\C)$ in the form of block matrices 
\begin{equation*}
S = \begin{pmatrix} 
 A & B \\ C & D 
\end{pmatrix}
\end{equation*}
where $A, B, C, D $ are $ n\times n $ complex matrices. It can be easily seen by using the equation $S^tJS =J$ that $S$ is symplectic if and only if $A^tC$ and $B^tD$ are symmetric and $A^tD-C^tB= I_n$. \\
  
For any $2n\times 2n$ complex block matrix  
\begin{equation*}
S = \begin{pmatrix} 
 A & B \\ C & D 
\end{pmatrix}
\end{equation*}  
define the following matrix $\Phi_S$ by
\begin{equation*}
\Phi_S(Z) = S(Z)= (AZ+B)(CZ+D)^{-1}
\end{equation*} 
where  $Z$ is a $n\times n$ complex matrix. This map is a generalization of a fractional linear transformation on complex plane. This map is well defined only if $CZ+D$ is invertible. \\
  
In \cite{F, Sie}, it is shown that if  $S \in SP_{2n}(\R)$ and $Z \in \mathcal S_n$ then  $\Phi_S $ is well defined and it maps the Siegel upper half space to itself. Since $SP_{2n}(\R)$ is a group, the map  $ \Phi_S: \mathcal S_n \rightarrow \mathcal S_n $ has been considered as a group action of $SP_{2n}(\R)$ on $\mathcal S_n$. However, this is not true if  $ S\in SP_{2n}(\C)$. 
  
\begin{exa}
Let $S =  \begin{pmatrix}  I_n & iI_n\\ iI_n & O \end{pmatrix}$.  Notice that $S\in SP_{2n}(\C)$ and $iI_n \in \mathcal S_n $.  Therefore 
\begin{equation*}
\Phi_S(iI_n)= (iI_n +iI_n)(iI_niI_n)^{-1} = -2iI_n\ .
\end{equation*}
Thus $\Phi_S(iI_n) \notin \mathcal S_n$ .
\end{exa}

Let $A(2n, \C) $ be the space of all matrices $S$ that satisfy $S^tJS = -J$, that is $S$ is antisymplectic. Define the following space:
\begin{equation*}
\overline{\mathcal S}_n = \{ Z\in \C^{d\times d}: Z = X +iY,  \,\, X^t=X, Y^t =Y, \,\, Y< 0 \}\ .
\end{equation*}

Following the methods of \cite{F}, it is easy to see that if $S\in A(2n, \R)$, $\Phi_S$ maps the Siegel half space to the space $\overline{\mathcal S}_n$. We will show this in Section 2.  Since the product of antisymplectic matrices is a symplectic matrix, the set $A(2n, \C)$ is not a group. Thus $\Phi_S $ is not a group action for such $S$. \\
   
Again, if we consider $S$ to be antisymplectic complex matrix,  $ \Phi_S $  does not map $\mathcal{S}_n$ to  $\overline{\mathcal{S}}_n$.
   
\begin{exa}  Let $S =  \begin{pmatrix}  iI_n & O\\ O & iI_n \end{pmatrix}  $  then $S\in A(2n, \C)$ and $iI_n \in \mathcal S_n$.  However $\Phi_S(iI_n) \notin \overline{\mathcal S}_n $ since
\begin{equation*}
\Phi_S(iI_n)= (iI_niI_n)(iI_n)^{-1} = iI_n \ .
\end{equation*}
\end{exa} 
    
One of the main goals in this paper is to partially classify all complex symplectic matrices for which the map $ \Phi_S: \mathcal S_n \rightarrow \mathcal S_n $ is well defined on $\mathcal S_n.$

We also want to consider a metric on  $\mathcal S_n$ as a generalization of the hyperbolic metric 

\begin{equation}\label{hm}
ds = \frac{\sqrt{dx^2 +dy^2}}{y}
\end{equation}
on the complex hyperbolic plane $\C^+.$  For this we consider the Finsler metric as in \cite{Froese} as follows.
   
Let $Z= X+iY \in \mathcal S_n $ and let the complex symmetric matrix $W$ be an element of the tangent space at $Z$. Consider the Finsler norm $F_Z$ defined by
\begin{equation*}
F_Z(W) = \| Y^{-\frac{1}{2}} W  Y^{-\frac{1}{2}}\|
\end{equation*} 
where $ \|\cdot\|$ is the usual matrix (operator) norm. This Finsler norm defines a metric on $\mathcal S_n$ as
\begin{equation*}
d_{\infty}(Z_1, Z_2) = \inf_{Z(t)} \int _0^1 F_{Z(t)}(\dot{Z}(t))\ dt
\end{equation*}
where the infimum is taken over all differentiable paths $Z(t)$ joining $Z_1$ to $Z_2.$
   
In the case when $n=1$, the Siegel upper half space $\mathcal S_1 = \C^+$ and the metric $d_{\infty}$ is same as the the hyperbolic metric defined in equation (\ref{hm}) on $\C^+$.  The length of a curve in the Euclidean plane is measured by $\sqrt{dx^2 +dy^2}$.   Similarly, the length of a curve $\gamma(t) = x(t) +i y(t)$, $t\in [0,1]$ on the hyperbolic plane $\C^+$ is defined by

\begin{equation*}
h(\gamma) = \int _0^1 \frac{1}{y(t)}\sqrt{\left(\frac{dx}{dt}\right)^2 + \left(\frac{dy}{dt}\right)^2}\ dt\ .
\end{equation*}
The hyperbolic distance between two points $z, w \in C^+$ is defined by 

\begin{equation*}
\rho(z,w) = \inf h(\gamma)
\end{equation*}
where the infimum is taken over all piecewise differentiable curves joining $z$ and $w$.  So for  $n=1$,  we have $d_{\infty}(z, w)= \rho (z, w)$.  Thus the Finsler metric $d_{\infty}$ on $\mathcal S_n$ is a generalization of the hyperbolic metric $\rho$ on $\C^+$.
    
We also want to define a metric on $\mathcal \overline{\mathcal S}_n$.  First note that the complex conjugate of a given matrix $Z= X+iY \in\overline{\mathcal S}_n$ is a matrix $\overline{Z} = X-iY \in \mathcal S_n$.   A metric $d_-$ on $ \overline{\mathcal S}_n $ is defined by  
\begin{equation*}
d_-(Z_1, Z_2) = d_{\infty}(\overline{Z}_1, \overline{Z}_2 )\ .
\end{equation*}
 Thus $(\mathcal S_n,d_{\infty})$ and $(\overline{\mathcal S}_n, d_-)$ are metric spaces.
  
\section{Preliminaries and Results}
 As mentioned above,  it is shown in \cite{F} that for each real symplectic matrix $S$, the map $\Phi_S$ maps the Siegel upper half space to itself.  However, it is not true when $S$ is real antisymplectic matrix.  In fact, it maps $\mathcal{S}_n$ onto $\overline{\mathcal{S}}_n$.  We have the following proposition.

\begin{prop} \label{p} For $S\in A(2n, \R) $ the map $\Phi_S $ maps $\mathcal S_n$ onto $\overline{\mathcal S}_n$. 
\end{prop}

\begin{proof}  
We show that $CZ+D$ is invertible for any $Z \in \mathcal S_n $ and $\Phi_S(Z)=(AZ+B)(CZ+D)^{-1} \in \overline{\mathcal S}_n .$  \\ Set $E=AZ+B$ and $F=CZ+D$. First we show that $F$ is invertible. We can write $\Phi_S(Z)$ as 

\begin{equation*}
S \begin{pmatrix} Z\\I\end{pmatrix} = \begin{pmatrix} E\\F\end{pmatrix}\ .
\end{equation*} 
From which we have,    
   
\begin{equation}\label{lab2}
\begin{aligned}
\frac{1}{2i}(E^*\,\, F^*) J \begin{pmatrix} E\\F\end{pmatrix} & = \frac{1}{2i}(Z^*\,\, I_n)S^tJS \begin{pmatrix} Z\\I_n\end{pmatrix}  = \frac{1}{2i}(Z^*\,\, I_n) (-J)\begin{pmatrix} Z\\I_n\end{pmatrix} \\ 
&= \frac{1}{2i}(Z^*\,\, I_n) \begin{pmatrix} O & -I_n\\ I_n  & O\end{pmatrix} \begin{pmatrix} Z\\I_n\end{pmatrix} \\  
&=\frac{1}{2i}(Z-Z^*) > 0\ .
\end{aligned}
\end{equation} 	
It follows that,	 

\begin{equation}\label{lab4}
\frac{1}{2i}(E^*F-F^*E)>0 \ .
\end{equation}
To see $F$ invertible, suppose $v$ is a solution of $Fv=0$ then we have $v^*F^*= 0$	and 

\begin{equation*}
v^*(E^*F-F^*E)v=0\ .
\end{equation*} 
From equation (\ref{lab4}), we have $v=0$ which implies $F$ invertible.
    	
Next we show that $\Phi_S (Z)$ is symmetric and $\operatorname{Im} \Phi_S (Z) <0$.  We have

\begin{equation*}
\begin{aligned}
(E^t\,\, F^t) J \begin{pmatrix} E\\F\end{pmatrix} & = (Z^t\,\, I_n)S^tJS \begin{pmatrix} Z\\I_n\end{pmatrix} = (Z^t\,\, I_n) (-J)\begin{pmatrix} Z\\I_n\end{pmatrix} \\ 
&= Z-Z^t = O\ .
\end{aligned}
\end{equation*}
This implies that $F^tE= E^tF$.  From this we get, 

\begin{equation*}
(EF^{-1})^t   =  (F^{-1})^tE^t  =(F^{-1})^t F^t E F^{-1}  = EF^{-1}
\end{equation*}
which means $\Phi_S (Z) $ is symmetric.   Moreover we have, 

\begin{equation*}
\begin{aligned} 
\operatorname{Im} \Phi_S (Z)= \operatorname{Im} (EF^{-1}) & = \frac{1}{2i}(EF^{-1}- (EF^{-1})^*) \\ &= \frac{1}{2i} (F^{-1})^*F^*(EF^{-1}- (EF^{-1})^*)F F^{-1} \\ &= \frac{1}{2i} (F^{-1})^*(F^*E- E^*F) F^{-1} \ .
\end{aligned}
\end{equation*}
Thus $\frac{1}{2i}(F^*E- E^*F) < 0$.  Now for any $v \in \C^n$ and $F^{-1}v\in \C^n$ we have

\begin{equation*}
\begin{aligned}
&v^*\frac{1}{2i} (F^{-1})^*F^*(EF^{-1}- (EF^{-1})^*)F F^{-1}v \\
&=(F^{-1}v)^* \frac{1}{2i} (F^*(EF^{-1}- (EF^{-1})^*)F) F^{-1}v \\ 
&=(F^{-1}v)^* \frac{1}{2i} (F^*E- E^*F) (F^{-1}v) < 0 \ .
\end{aligned}
\end{equation*}
It follows that $ \operatorname{Im} (EF^{-1}) < 0$.  Hence $\Phi_S$ maps $\mathcal S_n$ to $\overline{\mathcal S}_n$. 

Finally we show the map $\Phi_S $ is onto. Suppose $Z= X+iY \in \overline{\mathcal S}_n.$  We have 

\begin{equation*}
{S_Z}_- = \begin{pmatrix} - \sqrt{-Y} & X \sqrt{-Y^{-1}}\\ O & \sqrt{-Y^{-1}} \end{pmatrix} \in A(2n, \R)
\end{equation*} 
and  $\Phi_{{S_Z}_-}(iI_n) = Z$. $-Y$ is a symmetric positive definite matrix, therefore the Spectral Theorem guarantees that $\sqrt{-Y}$ and $\sqrt{-Y^{-1}}$ exist.  This completes the proof.
\end{proof}

As a consequence of this proposition we have the following theorem.
    
   
\begin{theorem}
If $S\in SP_{2n}(\C)$ is purely imaginary then $\Phi_S$ maps Siegel half space to the space $\overline{\mathcal S}_n$.
\end{theorem}
   
\begin{proof}
Let $S = iQ$, where 

\begin{equation*}
Q = \begin{pmatrix}  A & B \\ C & D \end{pmatrix}\ .
\end{equation*} 
Since $S^tJS = J$, we have $Q^tJQ = -J$, hence $Q \in A(2n,\R)$.  Moreover we have 

\begin{equation*}
\begin{aligned}
\Phi_{iQ}(Z) & = (iAZ+iB)(iCZ+iD)^{-1}\\
   	 & =(AZ+B)(CZ+D)^{-1} \\ 
   	 & =\Phi_{Q}(Z)\ .
\end{aligned}
\end{equation*}
Since $Q \in A(2n,\R)$,  by Proposition \ref{p} $\Phi_{Q}$ maps $\mathcal S_n$ to $\overline{\mathcal S}_n$ and hence $\Phi_{iQ}$ maps $\mathcal{S}_n$ to $\overline{\mathcal{S}}_n$.  Thus completing the proof. 
\end{proof}
   
In the following theorem we partially classify the matrices from $  SP_{2n}(\C)$ so that the action is well defined.
  
\begin{theorem}
Let
\begin{equation*}
S =  \begin{pmatrix}  A & B\\ C & D \end{pmatrix}  \in SP_{2n} (\C)
\end{equation*}
and $ Z \in \mathcal S_n$.  If $i(S^*JS-J)\ge 0$ then $\Phi_S(Z) = S(Z)= (AZ+B)(CZ+D)^{-1} \in \mathcal S_n$.
\end{theorem}
   
\begin{proof} 
For convenience set $E = AZ+B$ and $F=CZ+D$.  Suppose $i(S^*JS -J) \ge 0 $. For any $ Z \in \mathcal S_n$ we have 

\begin{equation*}
(Z^*\,\, I_n) \left(i(S^*JS-J) \right) \begin{pmatrix} Z \\ I_n  \end{pmatrix} \ge 0\ .
\end{equation*}  

On the other hand

\begin{equation*}
\begin{aligned}
&(Z^*\,\, I_n) \left(i(S^*JS-J) \right) \begin{pmatrix} Z \\ I_n  \end{pmatrix}\\
&=i\left(\left(S \begin{pmatrix} Z \\ I_n \end{pmatrix} \right)^*\left(S\begin{pmatrix} Z \\ I_n \end{pmatrix}\right) - (Z^*\,\, I_n)J\begin{pmatrix} Z \\ I_n \end{pmatrix}\right) \\
&= i(E^*\,\, F^*)J\begin{pmatrix} F \\ -E\end{pmatrix} - i(Z^*\,\, I_n)\begin{pmatrix} I_n \\ -Z\end{pmatrix} \\
&=i(E^*F - F^*E) - \frac{1}{i}(Z - Z^*)\ .
\end{aligned}
\end{equation*}
Therefore we have $i(E^*F - F^*E) +i(Z - Z^*) \ge 0$ and thus

\begin{equation*}
\frac{1}{i}(F^*E - E^*F) \ge \frac{1}{i}(Z-Z^*) = 2\textrm{Im}(Z)>0
\end{equation*}
since $Z\in\mathcal{S}_n$.  To show that $S(Z)\in\mathcal{S}_n$ we need to show that $F$ is invertible, $S(Z)$ is symmetric, and $\textrm{Im}(S(Z))>0$.  First we show that $F$ is invertible.  Suppose $Fv = 0$ for $v\in \C^n$, then it follows that $v^*(F^*E-E^*F)v/{2i} = 0$.   Since $-i(F^*E - E^*F)>0$, we have $v=0$.  Hence $F^{-1}$ exists and $EF^{-1}$ is well-defined.

Next we show symmetry.  Notice that since $S$ is a symplectic matrix and $Z\in\mathcal{S}_n$, we have

\begin{equation*}
\begin{aligned}
E^tF-F^tE &= \begin{pmatrix} Z \\ I_n \end{pmatrix}^tS^tJS\begin{pmatrix} Z \\ I_n \end{pmatrix} = \begin{pmatrix} Z \\ I_n \end{pmatrix}^tJ\begin{pmatrix} Z \\ I_n \end{pmatrix} = Z^t - Z = O\ .
\end{aligned}
\end{equation*}

Therefore $E^tF = F^tE$ and thus by the invertibility of $F$ we have

\begin{equation*}
(EF^{-1})^t = (F^{-1})^tE^t = (F^{-1})^tF^tEF^{-1} = EF^{-1}\ ,
\end{equation*}
thus showing the symmetry of $S(Z)$.  Finally a similar calculation to the one at the beginning of the proof shows that

\begin{equation*}
\textrm{Im}(S(Z)) = (F^{-1})^*\frac{1}{2i}\left(F^*E - E^*F\right)F^{-1}
\end{equation*}
and since $-i(F^*E - E^*F)>0$, it follows that $\textrm{Im}(S(Z))>0$.  This completes the proof.
\end{proof}

It should be noted the other direction may or may not be true.  One of the many issues in showing the other direction is going from $n\times n$ complex matrices to $2n\times 2n$ complex matrices.  The direction proven is simpler since the dimensions are going down.  In fact, one can give necessary and sufficient conditions on a self-adjoint block matrix  when it is positive definite.  In fact \cite{BV} stated without proof, and \cite{Gallier} proved in unpublished work, conditions on when a symmetric block matrix is positive definite.  This proof can be easily modified to allow for generic self-adjoint block matrices.  We state it without proof.

\begin{prop}
Let $M$ be a $2n\times 2n$ self-adjoint block matrix of the form:
\begin{equation*}
M = \begin{pmatrix} \alpha & \beta \\ \beta^* & \gamma \end{pmatrix}\ . 
\end{equation*}
The following are equivalent:
\begin{enumerate}
\item $M\ge 0$
\item $\alpha\ge 0$,\quad $(I_n - \alpha\alpha^\dag)\beta = O$,\quad and $\gamma - \beta^*\alpha^\dag\beta \ge 0$
\item $\gamma \ge 0 $,\quad $(I_n - \gamma\gamma^\dag)\beta^* = O$,\quad and $\alpha - \beta\gamma^\dag\beta^* \ge 0$
\end{enumerate}
where $P^\dag$ is the pseudo-inverse of $P$.
\end{prop}

For more details on pseudo-inverses of square matrices see \cite{Gallier1}.  Let $M = i(S^*JS-J)$, then it's clear that $M$ is self-adjoint.  A computation shows that

\begin{equation*}
M = \begin{pmatrix}
i(A^*C-C^*A) & i(A^*D-C^*B-I_n) \\ i(B^*C-D^*A+I_n) & i(B^*D-D^*B)
\end{pmatrix} :=\begin{pmatrix} \alpha & \beta \\ \beta^* & \gamma\end{pmatrix}
\end{equation*}

Based on the discussion from above, we know since $S$ is symplectic, then $A^tC$ and $B^tD$ are symmetric and $A^tD-C^tB= I_n$.  Moreover in this case, the entries of $S$ can not be all real, otherwise, we are the case already proven by \cite{F}.  Based on this proposition, for $M\ge 0$, we must have

\begin{equation*}
\begin{aligned}
&(1)\quad i(A^*C-C^*A)\ge 0 \\
&(2)\quad [I_n + (A^*C-C^*A)(A^*C-C^*A)^\dag](A^*D-C^*B-I_n) = O \\
&(3)\quad i(B^*D-D^*B) + \\
&\qquad +i(B^*C-D^*A+I_n)(A^*C-C^*A)^\dag(A^*D-C^*B-I_n) \ge 0\ .
\end{aligned}
\end{equation*}
Whether or not these conditions are equivalent to $i(S^*JS- J)\ge 0$ remains unknown.

The set of all matrices $S\in SP_{2n}(\C)$  for which the action is well defined is a subgroup of  $SP_{2n}(\C)$.   This was stated in \cite{Froese} without proof, therefore we provide a proof.

\begin{prop}\label{r1}
If $S$ and $R$ are any  $2n\times 2n$ complex matrices then $\Phi_S \circ \Phi_{R}  = \Phi_{SR}$ 
\end{prop}
  
\begin{proof}
Let

\begin{equation*}
S = \begin{pmatrix}  A & B \\ C & D \end{pmatrix} \textrm{ and } R = \begin{pmatrix}  E & F \\ G & H \end{pmatrix}\ .
\end{equation*}
For any $Z \in \mathcal S_n,$ for which $\Phi_S(Z)$ and $\Phi_R (Z) $ are well defined, we have 

\begin{equation*}
\begin{aligned}
\Phi_S \circ\Phi_R (Z) & = \Phi_S (\Phi_R (Z)) = \Phi_S ((EZ+F)(GZ+H)^{-1}) \\ 
& = \left(A(EZ+F)(GZ+H)^{-1} + B\right)\\
&\qquad\times \left(C(EZ+F)(GZ+H)^{-1} +D \right)^{-1} \\ 
&= \left(A(EZ+F)(GZ+H)^{-1} +B\right)(GZ+H) \\
&\qquad\times (GZ+H)^{-1}\left(C(EZ+F)(GZ+H)^{-1} +D \right)^{-1} \\ 
& = \left(A(EZ+F)+B(GZ+H)\right)\\
&\qquad\times\left(C(EZ+F) +D(GZ+H)\right)^{-1}\\ 
&= \left((AE+BG)Z+(AF+BH)\right)\\
&\qquad\times\left((CE+DG)Z+ (CF+DH))\right)^{-1} \\ 
& = \Phi_{SR}(Z)\ .
\end{aligned}
\end{equation*}
This works for an arbitrary $Z$, thus completing the proof.
\end{proof}
 
This next proposition was due to \cite{F}.
\begin{prop} 
The action of $\Phi_S$ on $\mathcal S_n$ is transitive.
\end{prop}
 
\begin{proof} 
We show that $\mathcal S_n$ has a single orbit. In other words, every $Z \in \mathcal S_n$ is an image of some $S\in SP_{2n}(\C)$.  Let $Z= X+iY \in \mathcal S_n$.  Define $S_Z$ by

\begin{equation*}
S_Z = \begin{pmatrix}  \sqrt{Y} & X \sqrt{Y^{-1}}\\ O & \sqrt{Y^{-1}} \end{pmatrix}\ .
\end{equation*}
Recall $Y$ is a symmetric positive definite matrix, therefore the Spectral Theorem guarantees that $\sqrt{Y}$ and $\sqrt{Y^{-1}}$ exist.   It is easy to verify that $S_Z^tJS_Z = J$. Therefore, $S_Z\in SP_{2n}(\C)$. Now consider $S_Z(iI_n)$.  Indeed we have 

\begin{equation*}
S_Z(iI_n) = (\sqrt{Y} iI_n + X \sqrt{Y^{-1}} ) (\sqrt{Y^{-1}})^{-1} =  X+iY =Z\,
\end{equation*}
thus showing $\mathcal{S}_n$ has a single orbit completing the proof.
\end{proof}

Since the action of 
\begin{equation*}
I_{2n} := \begin{pmatrix} I_n & O\\ O & I_n \end{pmatrix}
\end{equation*}
is same as the action of 
\begin{equation*} 
-I_{2n} = \begin{pmatrix} -I_n & O\\ O & -I_n \end{pmatrix}
\end{equation*}
we can identify them so that we can form the quotient space 

\begin{equation*}
SP_{2n}(\R)/\{I_{2n}, -I_{2n}\}
\end{equation*}
again denoted by $SP_{2n}(\R).$ This makes the action of $SP_{2n}(\R)$ on $\mathcal S_n$ bijective.  The following proposition was stated in \cite{Froese} again without proof.  Like before we present a proof for completeness of this paper.
  
\begin{prop}
The set of matrices 
\begin{equation*}
K= \{U \in SP_{2n}(\R): \Phi_U(iI_n) = iI_n\}
\end{equation*}
is precisely the subgroup of orthogonal matrices in  $SP_{2n}(\R)$. 
\end{prop}

\begin{proof}
Suppose 
\begin{equation*}
U = \begin{pmatrix}  A & B \\ C & D \end{pmatrix}  \in SP_{2n}(\R)
\end{equation*}
such that $\Phi_U(iI_n) = iI_n$ then $(A iI_n +B)(C iI_n +D)^{-1} =  iI_n$.  This implies that $iA+B =  iI_n (iC+D)$, therefore it follows that $B= -C$ and $A=D$.   Hence $ U$ and $U^t$ are of the forms

\begin{equation*}
U= \begin{pmatrix}  A & B \\ -B & A \end{pmatrix} \textrm{ and }U^t = \begin{pmatrix}  A^t & -B^t \\ B^t & A^t \end{pmatrix}\ .
\end{equation*} 
Thus we have

\begin{equation*}  		
\begin{aligned}
UU^t & = \begin{pmatrix}  A & B \\ -B & A \end{pmatrix} \begin{pmatrix}  A^t & -B^t \\ B^t & A^t \end{pmatrix} \\ 
&= \begin{pmatrix}  AA^t +BB^t & -AB^t+BA^t \\ -BA^t+AB^t & BB^t +AA^t \end{pmatrix} = I_{2n}
\end{aligned}
\end{equation*}
since $AA^t +BB^t = I_n$ and $-AB^t$ is symmetric.   This shows that $ U^t= U^{-1}$, hence $U$ is orthogonal.  Moreover if
  	 
\begin{equation*}  	 
U = \begin{pmatrix}  A & B \\ -B & A \end{pmatrix}\in K\ ,
\end{equation*}
then $U^{-1} \in K$.   This is true since

\begin{equation*}
\begin{aligned}
\Phi_{U^{-1}}(iI_n) & = \begin{pmatrix}  A^t & -B^t \\ B^t & A^t \end{pmatrix}(iI_n) = \left( A^t (iI_n)  -B^t \right) \left( B^t (iI_n) + A^t \right)^{-1} \\
&= ( Ai -B )^t \left[( B(iI_n) + A)^t \right]^{-1} = \left[(Bi + A)^{-1}(Ai -B)\right]^t \\ 
&= \left[i(Bi + A)^{-1}(A + Bi) \right]^t =  iI_n\ .
\end{aligned}
\end{equation*}
In addition, if $U$ and $V$ are both in $K$ then by Proposition \ref{r1} we have $\Phi_{UV^{-1}}(iI_n)  =\Phi_U(\Phi_{V^{-1}}(iI_n))  = \Phi_U(iI_n)) = iI_n$.  This shows that $K$ is a subgroup.  Thus the result follows.
\end{proof}
  	
Since $K$ is the subgroup of orthogonal matrices in $SP_{2n}(\R)$ general theory tells us that $K$ is a normal subgroup.  Therefore we can consider the quotient space 
\begin{equation*}
SP_{2n}(\R)/K = \{SK : S\in SP_{2n}(\R)\}
\end{equation*} 
where $SK= \{SU : U \in K\}$ is a coset of $K$.   As shown in \cite{F}, this space is metrizable with a metric given by

\begin{equation*}
d^{S}(S_1K, S_2K) = 2 \ln \| S_1^{-1}S_2\|\ .
\end{equation*}
Notice that the metric is independent of the choice of representative in the equivalence class $SK$.
 
Define a map $\Psi$ in the following way

\begin{equation}\label{psi_isometry}
\Psi: SP_{2n}(\R)/K\rightarrow \mathcal{S}_n \textrm{ via } SK\mapsto \Phi_S(iI_n)\ .
\end{equation}
Since $\Phi_S$, when acting $SP_{2n}(\R)$, is a bijective map, the map $\Psi$ is also bijective.   The following theorem was proved in \cite{Froese}.
  		
\begin{theorem}	
The map $\Psi$, defined in equation (\ref{psi_isometry}), is an isometry for the metrics $d^{S}$ and $d_{\infty}$, that is 

\begin{equation*}
d^{S}(S_1 K, S_2 K) = d_{\infty}(S_1(iI_n), S_2(iI_n))\ .
\end{equation*} 
In particular, 

\begin{equation*} 
d_{\infty}(Z_1, Z_2)  = 2 \ln \|{S_{Z_1}} ^{-1} S_{Z_2}\|\ .
\end{equation*} 
\end{theorem}
We skip the proof for brevity.  Based on this result we have the following theorem.
  
\begin{theorem} 
If $S\in SP_{2n}(\C) $ is purely imaginary then the map $\Phi_S$ is an isometry between the metric spaces $(\mathcal S_n,d_{\infty})$ and $(\overline{\mathcal S}_n, d_-)$.
\end{theorem}
  	
\begin{proof}
Let $S = iQ, $ where 

\begin{equation*}
Q = \begin{pmatrix}  A & B \\ C & D \end{pmatrix}\ .
\end{equation*}
Since $S^tJS = J$ we have $Q^tJQ = -J$, hence $Q \in A(2n, \R)$.   For any $Z= X+iY \in \mathcal S_n$ we have $\Phi_{S_Z}(iI_n)= S_Z(iI_n)= Z$.  Let 

\begin{equation*}
{S_Z}_- = \begin{pmatrix} - \sqrt{Y} & X \sqrt{Y^{-1}}\\ O & \sqrt{Y^{-1}} \end{pmatrix} \textrm{ and } I_- = \begin{pmatrix}  -I_n & O\\ O & I_n \end{pmatrix}\ .
\end{equation*}
Then  ${S_Z}_- = S_Z I_-$ and ${S_Z}_-^tJ{S_Z}_- = -J$ implying that $ {S_Z}_- \in A(2n,\R)$.  Since both $Q$ and ${S_Z}_-$ are in $A(2n, \R)$ this implies that  $Q{S_Z}_-$ is in $SP_{2n}(\R)$.   Notice that $ {S_Z}_-(iI_n)= \overline{Z}$.  Therefore we have		

\begin{equation*}  	    	  	 
\begin{aligned}
d_-(\Phi_S(Z), \Phi_S(W)) & =  d_-(\Phi_{iQ}(Z),\Phi_{iQ}(W) =  d_-(\Phi_Q(Z),\Phi_Q (W)) \\ 
& = d_{\infty}(\overline{\Phi_Q (Z)}, \overline{\Phi_Q (W)}) = d_{\infty}(\Phi_Q(\overline{Z}),\Phi_Q (\overline{W}))\\ 
& = d_{\infty}(Q{S_Z}_-(iI_n), Q{S_W}_-(iI_n))\\ 
&= 2\ln\|(Q {S_Z}_-)^{-1} Q{S_W}_- \| = 2\ln\|{S_Z}_-^{-1}{S_W}_- \|\\ 
&=  2\ln\|I_- {S_Z} ^{-1}{S_W} I_- \| =   2\ln\|{S_Z} ^{-1} {S_W}\| \\ 
&=d_{\infty}(Z, W)
\end{aligned}
\end{equation*}
since $I_-$ is a unitary.  Thus this completes the proof.
\end{proof}
  
Unfortunately in general, the map $\Phi_S$ is not an isometry.  In fact we have the following theorem.   
   	
\begin{theorem}
If the matrix $S$ is simplectically similar to 

\begin{equation*}
R = \begin{pmatrix}  I_n & Z\\ I_n & 0 \end{pmatrix} \in SP_{2n} (\C)
\end{equation*} 
where $ Z\in \mathcal S_n$, that is $S = PRP^{-1} $ for some $P\in SP_{2n} (\R) $ then there exists a $\delta$ with $0<\delta<1$ such that for all $Z_1,Z_2\in\mathcal{S}_n$,

\begin{equation*}
d_{\infty}(S(Z_1), S(Z_2)) \leq \delta d_{\infty}(Z_1, Z_2)\ .
\end{equation*}
\end{theorem}
  
\begin{proof}
Claim: for $Z= X+iY, W=W_1 +i W_2 \in \mathcal S_n$ 

\begin{equation*}
F_{W+Z}(\gamma ) < \delta F_W(\gamma)
\end{equation*}
for some $0 < \delta < 1$. Indeed we have 

\begin{equation*}
\begin{aligned}
F_{W+Z}(\gamma) & = F_{(W_1 + X )+i(W_2+Y)}(\gamma ) = \| (W_2+Y)^{-\frac{1}{2}} \gamma (W_2+Y)^{-\frac{1}{2}}\| \\ 
& = \| (W_2+Y)^{-\frac{1}{2}} W_2^{\frac{1}{2}}W_2^{-\frac{1}{2}}\gamma W_2^{-\frac{1}{2}}W_2^{\frac{1}{2}}(W_2+Y)^{-\frac{1}{2}} \| \\ 
&\le \|(W_2+Y)^{-\frac{1}{2}} W_2^{\frac{1}{2}} \|^2 \| (W_2^{-\frac{1}{2}} \gamma  W_2^{-\frac{1}{2}}\|\ .
\end{aligned}
\end{equation*}
Since $Y>0$, we have  $W_2^{\frac{1}{2}} < (W_2+Y)^{\frac{1}{2}}$ and $ (W_2+Y)^{-\frac{1}{2}}>0$.  It follows

\begin{equation*}
(W_2+Y)^{-\frac{1}{2}} W_2^{\frac{1}{2}} <  (W_2+Y)^{-\frac{1}{2}}(W_2+Y)^{\frac{1}{2}} = I_n\ .
\end{equation*} 
Therefore we have $\|(W_2+Y)^{-\frac{1}{2}} W_2^{\frac{1}{2}}\| < 1$. Thus we can choose 

\begin{equation*}
\delta = \| (W_2+Y)^{-\frac{1}{2}} W_2^{\frac{1}{2}}\|
\end{equation*}
proving the claim.   Now let $\gamma(t)$ be a path joining $Z_1, Z_2$ so that $\gamma(t)+Z$ is a path joining $Z_1+Z, Z_2+Z$ and

\begin{equation*}
\int_0^1 F_{{\gamma(t)}+Z}(\dot{\gamma}(t))\ dt<\delta\int_0^1 F_{\gamma(t)}(\dot{\gamma}(t))\ dt\ .
\end{equation*}
Taking the infimum over all the paths $\gamma(t)$ joining $Z_1$ and $Z_2$, we get 

\begin{equation*}
\int_0^1 F_{\gamma(t)+Z}(\dot{\gamma}(t))\ dt< \delta d_{\infty}(Z_1, Z_2)\ . 
\end{equation*}
Again taking the infimum over all the paths $\gamma(t)+Z$ joining $Z_1 + Z$ and $Z_2 +Z$, we get 

\begin{equation*}
d_{\infty}(R(Z_1), R(Z_2)) \le \delta d_{\infty}(Z_1, Z_2)\ .
\end{equation*} 
   	
Finally we get

\begin{equation*}
\begin{aligned}
d_{\infty}(S(Z_1), S(Z_2)) &= d_{\infty}(PRP^{-1}(Z_1), PRP^{-1}(Z_2))\\
&= d_{\infty}(RP^{-1}(Z_1), RP^{-1}(Z_2)) \\
&\le \delta d_{\infty}(P^{-1}(Z_1), P^{-1}(Z_2)) \\
&= \delta d_{\infty}(Z_1, Z_2 )\ ,
\end{aligned}
\end{equation*}
completing the proof.
\end{proof}
  
Recall that the Finsler metric $d_\infty(\cdot,\cdot)$ becomes the hyperbolic metric when $n=1$.  Therefore we can compare distances of numbers in the upper half plane and corresponding counterparts in the Siegel space.  We end the paper with the following result.
\begin{theorem} 
Let $\mathbb{B}(0)$ be the unit ball in $\C^n$.  For any $Z_1, Z_2 \in \mathcal S_n$ and $v\in \mathbb B(0)$, 

\begin{equation*}
d_{\infty}(v^*Z_1v, v^*Z_2v) \leq d_{\infty}(Z_1, Z_2)\ .
\end{equation*}
\end{theorem}

\begin{proof}
If $Z(t)$ is a path joining $Z_1$ and $Z_2$ then $v^*Z(t)v$ is a path joining $v^*Z_1v$ and $v^*Z_2v$ with $v\in \mathbb B(0).$  For $Z(t) = X(t)+iY(t)$ and $v \in \mathbb B(0)$ we have 

\begin{align*}
\operatorname{Im}(v^*Z(t)v) & = \frac{1}{2i} (v^*Z(t)v - (v^*Z(t)v)^*)
\\ & =\frac{1}{2i} v^*(Z(t) - Z(t)^*) v^* \\ & = v^*Y(t) v^*
  \end{align*}
with $v^*Y(t)v >0$.  Recall that for a $n\times n$ matrix $A$ and a vector $x$ we have the following submultiplicative estimate: $\|Ax\|\le \|A\|\|x\|$.  Using this, the Cauchy-Schwarz inequality, and the fact that the Finsler metric becomes the hyperbolic metric in one dimension we get

\begin{equation*}  	
\begin{aligned}
F_{v^* Z(t)v}(v^*\dot{Z}(t) v) & = \| (v^*Yv)^{-\frac{1}{2}} (v^* \dot{Z}(t)v) (v^*Yv)^{-\frac{1}{2}} \| \\ 
& =\frac{1}{v^*Yv} \| v^*Y^{\frac{1}{2}} Y^{-\frac{1}{2}} \dot{Z}(t)Y^{-\frac{1}{2}} Y^{\frac{1}{2}}v\|\\ 
& =\frac{1}{v^*Yv} \langle Y^{\frac{1}{2}}v,  Y^{-\frac{1}{2}}\dot{Z}(t)Y^{-\frac{1}{2}} Y^{\frac{1}{2}}v\rangle \\ 
& \leq \frac{1}{v^*Yv} \|Y^{\frac{1}{2}}v \|^2 \|Y^{-\frac{1}{2}}\dot{Z}(t)Y^{-\frac{1}{2}}\| \\ 
&= \|Y^{-\frac{1}{2}}\dot{Z}(t)Y^{-\frac{1}{2}}\| = F_{Z(t)}(\dot{Z}(t))\ .
\end{aligned}
\end{equation*}
Upon integration we get,

\begin{equation*}
\int_0^1 F_{v^*Z(t)v}(v^*\dot{Z}(t)v)\ dt \le \int_0^1 F_{Z(t)}(\dot{Z}(t))\ dt\ .
\end{equation*}
As before, taking the infimum over all the paths $Z(t)$ joining $Z_1$ and $Z_2$, we get 

\begin{equation*}
\int_0^1 F_{v^*Z(t)v}(v^*\dot{Z}(t)v)\ dt \le  d_{\infty}(Z_1, Z_2)\ .
\end{equation*}
Finally taking the infimum over all the paths $v^*Z(t)v$ joining $v^*Z_1v$ and $v^*Z_2v$, we get

\begin{equation*}
d_{\infty}(v^*Z_1v, v^*Z_2v) \leq  d_{\infty}(Z_1, Z_2)\ .
\end{equation*}
This completes the proof.
\end{proof}

\end{document}